\documentclass{article}

\usepackage[centertags]{amsmath}
\usepackage{hyperref}
\usepackage{amsfonts}
\usepackage{amssymb}
\usepackage{amsthm}
\usepackage{newlfont}
\usepackage{amscd}
\usepackage{amsmath,amscd}
\usepackage[all]{xy}
\usepackage{verbatim}

\usepackage{color}

\newcommand{\QQ}{\mathbb{Q}}
\newcommand{\ZZ}{\mathbb{Z}}

\newcommand{\PP}{\mathbb{P}}

\newcommand{\AAA}{\mathbb{A}}
\newcommand{\be}{\begin{equation}}
\newcommand{\ee}{\end{equation}}

\def\sig{{\sigma}}

\newtheorem{thm}{Theorem}[section]

\newtheorem{lem}[thm]{Lemma}

\theoremstyle{definition}
\newtheorem{defn}[thm]{Definition}
\theoremstyle{remark}

\DeclareMathOperator{\Ker}{Ker}

\DeclareMathOperator{\Pic}{Pic}

\DeclareMathOperator{\Div}{Div}

\DeclareMathOperator{\Br}{Br}

\DeclareMathOperator{\Hom}{Hom}

\DeclareFontEncoding{OT2}{}{} 
\DeclareTextFontCommand{\textcyr}{\fontencoding{OT2}\fontfamily{wncyr}\fontseries{m}\fontshape{n}\selectfont}

\newcommand{\Sha}{\textcyr{Sh}}

\title{  On the Brauer-Manin Obstruction Applied to Ramified Covers.}

\author{Tomer M. Schlank}
\begin{document}
\maketitle
\setlength{\baselineskip}{18pt}

\section*{Abstract}
The Brauer-Manin obstruction is used to explain the failure of the
local-global principle for algebraic varieties.
In 1999 Skorobogatov gave the first example of a variety whose failure to
 satisfy that principle is not explained  by the
Brauer-Manin obstruction.
He did so by applying the Brauer-Manin obstruction to \'{e}tale covers of the
variety, thus defining a finer obstruction. In 2008 Poonen gave
the first example of failure of the local-global principle which cannot be
explained by Skorobogatov's \'{e}tale-Brauer obstruction.
However, Poonen's construction was not accompanied by a definition of a new finer obstruction. In this paper we present
a possible definition for such an obstruction by applying the
Brauer-Manin obstruction to some ramified covers as well, and show that this new obstruction can
in some cases explain Poonen counterexample over a totally imaginary number field.

\tableofcontents

\section{Introduction}
Call a variety $X$ nice if it is smooth, projective, and geometrically integral. Given a nice variety $X$ over  a global field  $k$, a major problem is to decide whether $X(k)= \emptyset$.  As a first approximation one can consider the set $X(\AAA_k) \supset X(k)$, where $\AAA_k$ is the adeles ring of $k$. It is a classical theorem of Minkowski and Hasse that if $X$ is a quadric then $X(\AAA_k)\neq \emptyset \Rightarrow X(k)\neq \emptyset$. When a variety $X$ satisfies this property we say that it satisfies the Hasse (or the Local-Global) principle. In the 1940's Lind and Reichardt (~\cite{Lin40}, ~\cite{Rei42} ) gave examples of genus 1 curves that do not satisfy the Hasse principle. More counterexamples to the Hasse principle were given throughout the years, until in 1971 Manin ~\cite{Man70} described a general obstruction to the Hasse principle, that explained all the examples that were known to that date. The obstruction (known as the Brauer-Manin obstruction) is defined by considering a certain set $X(\AAA_k)^{\Br} $,  $X(k) \subset X(\AAA_k)^{\Br}  \subset X(\AAA_k)$. If $X$ is a counterexample to the Hasse principle we say that it is accounted for or explained by the Brauer-Manin obstruction if $\emptyset = X(\AAA_k)^{\Br}  \subset X(\AAA_k) \neq \emptyset$.
\\
In 1999 Skorobogatov ~\cite{Sko99} defined a refinement of the Brauer-Manin obstruction (also known as the \'{e}tale-Brauer-Manin obstruction) and used it to produce an example of a variety $X$ such that $X(\AAA_k)^{\Br}  \neq \emptyset$ but $X(k) = \emptyset$. Namely, he described a set  $X(k) \subset X(\AAA_k)^{\acute{E}t,\Br}  \subset X(\AAA_k)^{\Br}  \subset X(\AAA_k)$ and found a variety $X$ such that $X(\AAA_k)^{\acute{E}t,\Br}  = \emptyset$ but $X(\AAA_k)^{\Br}  \neq \emptyset$.

In his paper from 2008 ~\cite{Poo08} Poonen constructed the first and currently only known example of a variety $X$ such that $X(\AAA)^{\acute{E}t,\Br}  \neq \emptyset$ but $X(k)=\emptyset$. However, Poonen's method of showing that  $X(k)=\emptyset$ relies on the details of his specific construction and is not explained by a new finer obstruction.
Therefore, one wonders if Poonen's counterexample can be accounted for by an additional refinement of $X(\AAA_k)^{\acute{E}t,\Br} $. Namely, can one give a general definition of a set
$$ X(k) \subset X(\AAA_k)^{new} \subset X(\AAA_k)^{\acute{E}t,\Br} $$
such that Poonen's variety $X$ satisfies $X(\AAA_k)^{new}= \emptyset$. In this paper we suggest such a refinement.

The results presented in this paper hold for global fields without real embeddings, i.e for function fields and totaly imaginary number fields, but we believe that this restriction is not essential.

The author would like to thank Jean-Louis Colliot-Th\'{e}l\`{e}ne and Alexei Skorobogatov for many useful discussions.

Most of the work presented here was done while attending at the "Diophantine equations" trimester program at Hausdorff Institute in Bonn. The author would like to thank the staff of the institute for providing a pleasant atmosphere and excellent working conditions.

The author would also like to thank Yonatan Harpaz for his useful comments on the first draft of this paper.

\section{Ramified Covers and the Brauer-Manin Obstruction}\label{sec:ram}
In ~\cite{Sko99} Skorobogatov presented the \'{e}tale-Brauer-Manin obstruction. In this section we shall present a slight generalization which will be applicable to our case.
\subsection{Twisting torsors and the \'{e}tale-Brauer-Manin obstruction}
Let $k$ be a global field, $G$ be a finite $k$-group and $X$ be a $k$-variety.
Recall that a $G$-torsor over $X$ is a map $\pi:Y \to X$ a together with a $G$-action on $Y$ respecting $\pi$ such that over $\overline{k}$ the action on the fibers of $\pi$ is free and transitive.

Now let $\pi:Y \to X$  be a $G$-torsor and $\sigma \in H^1(K,G)$,   $\sigma$ can represented by a right $G$ principal homogenous space $P_\sigma$. We denote $Y^{\sigma}:= P_\sigma \times^G Y$, note
that there is a natural map $\pi^{\sigma}:Y^{\sigma} \to X$  and that $\pi^{\sigma}:Y^{\sigma} \to X$ is  naturally a $G^{\sigma}$-torsor over $X$ where $G^{\sigma}$ is the suitable inner form of $G$ . We call   $\pi^{\sigma}:Y^{\sigma} \to X$ the \emph{twist of $\pi:Y \to X$ by $\sigma$}.

One of the  main attributes of torsors who make them useful in the study of rational points is the fact that
given any $\pi:Y \to X$ a $G$-torsor.

We have:

$$(*),\quad  X(k) = \biguplus \limits_{\sigma \in H^1(k,G)} \pi^{\sigma}(Y^{\sigma}(k)) $$

The definition of the \'{e}tale-Brauer-Manin obstruction applying the Brauer-Manin obstruction to torsors of $X$.
Namely , since $Y(k) \subset Y(\AAA)^{Br}$  for every $Y$ we have by (*):

$$ X(k) \subset  X(\AAA)^{\pi,Br}  := \biguplus \limits_{\sigma \in H^1(k,G)} \pi^{\sigma}(Y^{\sigma}(\AAA)^{Br}).$$

By taking all possible such torsos over $X$  we get:

$$ X(\AAA)^{\acute{E}t,Br} = \bigcap \limits_{\pi} X(\AAA)^{\pi,Br} $$

\subsection{Brauer-Manin obstruction applied to ramified covers}
In this subsection we define slight generalizations of the concepts of torsors and the \'{e}tale-Brauer-Manin obstruction, which we use in order to get a "stornger" obstruction then the \'{e}tale-Brauer-Manin obstruction.
\begin{defn}
Let $X$ be a geometrically integral variety over a field $k$, $G$ a finite $k$-group and $D \subset X$ an effective divisor. A  \emph{$G$ -quasi-torsor over $X$ unramified outside $D$ } is a map $\pi:Y \to X$ and a $G$-action on $Y$ respecting $\pi$ such that
\begin{enumerate}
\item $\pi$ is a surjective quasi-finite morphism of generic degree $|G|$.
\item $G$ acts on the generic fibre freely and transitively.
\item The ramification locus of $\pi$ is contain in $D$.
\end{enumerate}
We call $d = |G|$ the degree of $Y$.
\end{defn}

Now let $D$ be a divisor and $\pi:Y \to X$ be a $G$-quasi-torsor over $X$ unramified outside $D$. Note that like in the case of a usual $G$-torsor, given an element  $\sig \in H^1(k,G)$ one can twist $\pi:Y \to X$ by $\sig$  and get a $G^{\sig}$-quasi-torsor $\pi^\sig: Y^\sig \to X$.
Now if we assume that  $D(k)= \emptyset$ in similar way to (*) we get:
$$(**),\quad  X(k) = \biguplus \limits_{\sigma \in H^1(k,G)} \pi^{\sigma}(Y^{\sigma}(k)) $$

By  $(**)$ we get:

$$X(k) \subset  X(\AAA)^{\pi,\Br}  := \biguplus \limits_{\sigma \in H^1(k,G)} \pi^{\sigma}(Y^{\sigma}(\AAA)^{\Br}).$$
By taking all possible such torsos over $X$  unramified outside $D$ we get:

$$ X(\AAA)^{\acute{e}t,\Br \thicksim D} = \bigcap \limits_{\pi} X(\AAA)^{\pi,\Br} \subset X(\AAA)$$

When $X(\AAA)^{\acute{E}t,\Br \thicksim D} = \emptyset$ we shall say that \emph{the absence of rational points is explained by the $(\acute{E}t,\Br \thicksim D)$-obstruction}

In this paper we shall show (under some conditions) that for the variety  $X$ that Poonen defines in ~\cite{Poo08}, one can choose a divisor $D \subset X$ such that $D(k)=\emptyset$ and
$X(\AAA)^{et-Br\thicksim D} = \emptyset$. This gives an obstruction theoretic explanation of the absence of rational points on $X$.

\section{Conic bundles}
In this section we shall present a construction of conic bundles on a nice variety $B$ and study some of its properties. This construction appears in ~\cite{Poo08} \S 4 and Poonen used it in order to build his counterexample.   We base out notation here on his, and add some notations of our own.

Trough out the rest of the paper given a $k$-variety $X$  the corresponding base-change to $\overline{k}$ where $\overline{k}$ is an algebraic closure of $k$.

Let $k$ be any field of characteristic not $2$. Let $B$ be a nice $k$-variety. Let $\mathcal{L}$ be
a line bundle on $B$. Let $\mathcal{E}$ be the rank $3$ bundle sheaf
$$ \mathcal{O} \oplus \mathcal{L} \oplus \mathcal{L} $$
on B. Let $a \in  k^\times$ and let $s \in  \Gamma(B,\mathcal{L}^{\otimes 2})$ be a nonzero global section. Consider the section
$$ 1 \oplus (-a) \oplus (-s) \in \Gamma(B,\mathcal{O} \oplus \mathcal{O} \oplus \mathcal{L}^{\otimes 2}) \subset
\Gamma (B, Sym^2 \mathcal{E}) $$
(where the inclusion $\mathcal{O} \oplus \mathcal{O} \oplus \mathcal{L}^{\otimes 2}\subset Sym^2 \mathcal{E})$ is the the diagonal one)
The zero locus of  $1 \oplus (-a) \oplus (-s)$ in $\PP \mathcal{E}^v$ is a projective geometrically integral scheme $X=X(B,\mathcal{L},a,s)$ with a morphism $\alpha : X \to  B$.

We shall call
$$ (\mathcal{L},s,a)\in Div B \times \Gamma(B,\mathcal{L}^{\otimes 2}) \times k^{\times} $$
a \emph{conic bundle datum on $B$} and $X$ \emph{the total space of $(\mathcal{L},s,a)$}. We denote $X=Tot_B(\mathcal{L},s,a)$.

If $U$ is a dense open subscheme of $B$ with a trivialization $\mathcal{L}|_U \cong \mathcal{O}_U$ and we identify $s|_U$ with an element of $\Gamma(U,\mathcal{O}_U)$ then the affine scheme defined by $y^2 - az^2 = s|_U$ in $\AAA^2_U$ is a dense open subscheme of $X$. We therefore refer to $X$ as the conic bundle given by $y^2 - az^2 = s$.

In the special case where $B = \PP^1$, $\mathcal{L} = \mathcal{O}(2)$, and the homogeneous form $s \in \Gamma(\PP^1,\mathcal{O}(4))$ is
separable, $X$ is called the Ch\^{a}telet surface given by $y^2-az^2 = s(x)$, where $s(x) \in  k[x]$ denotes
a dehomogenization of $s$.

Returning to the general case, we let $Z$ be the subscheme $s = 0$ of $B$. We call $Z$ the degeneracy
locus of the conic bundle $(\mathcal{L},s,a)$. Each fiber of $\alpha$ above a point of $B-Z$ is a smooth plane conic, and each fiber above a geometric point of $Z$ is a union of two projective lines crossing transversally
at a point. A local calculation shows that if $Z$ is smooth over $k$ then $X$ is smooth over $k$.

\begin{lem} The generic fiber $\overline{X}_\eta$ of $\overline{X} \to  \overline{B}$ is isomorphic to $\PP^1_{\kappa(\overline{B})}$ where $\kappa(\overline{B})$ is the field of rational functions on $\overline{B}$ .
\end{lem}

\begin{proof} It is a smooth plane conic and it has a rational point since $a$ is a square in $\overline{k} \subset \kappa(\overline{B})$.
\end{proof}

\begin{lem}
Let $B$ be a nice $k$-variety and $(\mathcal{L},s,a)$ a conic bundle datum on $B$. Denote the corresponding bundle $\alpha:X \to B$ and the generic point of $B$ by $\eta$. Let $Z$ be the degeneracy locus. Assume that $\overline{Z}$ is the union of the irreducible components $\overline{Z} = \bigcup_{1\leq i\leq r} \overline{Z}_i$. Then there is a natural exact sequence of Galois modules.

$$
\xymatrix{
0 \ar[r] & \bigoplus\ZZ \overline{Z}_i  \ar[r]^-{\rho_1} & \Pic \overline{B} \oplus \bigoplus \ZZ \overline{Z}^{+}_i \oplus \bigoplus \ZZ \overline{Z}^{-}_i \ar[r]^-{\rho_2} & \Pic \overline{X}
 \ar[r]_-{\rho_3} \ar[rd]_{deg} &  \Pic {\overline{X}}_\eta \ar[r] \ar@{=}[d] \ar@/_/[l]_-{\rho_4} & 0 \\
 & &  &  &  \mathbb{Z}&
}
$$

where $\rho_4$  is a natural section of $\rho_3$.

\end{lem}

\begin{proof}
Call a divisor of $\overline{X}$ vertical if it is supported on prime divisors lying above prime divisors
of $\overline{B}$, and horizontal otherwise
Denote by $\overline{Z}^\pm_i$ the divisors that lie over $\overline{Z}_i$ and defined by the additional condition that $y=\pm \sqrt{a} z$. Now define $\rho_1$ by
$$ \rho_1(\overline{Z}_i) = (-\overline{Z}_i, \overline{Z}^{+}_i, \overline{Z}^{-}_i) $$
and $\rho_2$ by
$$ \rho_2 (M, 0, 0) = \alpha^* M $$
$$ \rho_2 (0, \overline{Z}^+_i, 0) = \overline{Z}^+_i $$
$$ \rho_2 (0, 0,\overline{Z}^-_i) = \overline{Z}^-_i $$
Let $\rho_3$ be the map induced by $\overline{X}_\eta \to \overline{X}$. Each $\rho_i$ is $\Gamma_k$-equivariant. Given a prime divisor $D$ on $\overline{X}_\eta$ we take $\rho_4(D) $ to be its Zariski closure in $\overline{X}$. It is clear that $\rho_3 \circ \rho_4 = Id$ and so $\rho_3$ is indeed surjective.

The kernel of $\rho_3$ is generated by the classes of vertical prime divisors of $X$. In fact, there is exactly one
above each prime divisor of $B$ except that above each $\overline{Z}_i \in  \Div \overline{B}$ we have both $\overline{Z}^+_i, \overline{Z}^-_i \in  \Div \overline{X}$. This proves exactness at $\Pic \overline{X}$.

Now, since $\alpha:\overline{X}\to \overline{B}$ is proper a rational function on $\overline{X}$ with a vertical divisor must be the pullback of a rational function on $\overline{B}$. Using the fact that the image of $\rho_2$ contain only vertical divisors, we prove exactness at
$$ \Pic \overline{B} \oplus \bigoplus \ZZ \overline{Z}^{+}_i \oplus \bigoplus \ZZ \overline{Z}^{-}_i $$
The injectivity of $\rho_1$ is trivial.
\end{proof}

\section{Poonen's Counterexample}
Poonen's construction can be done over any global field $k$ of characteristic different form $2$. We shall follow his construction in this section. Let $a \in k^\times$ and let $\tilde{P}_\infty(x), \tilde{P}_0(x) \in k[x]$ be   relatively prime separable degree $4$ polynomials such that the (nice) Ch\^{a}telet surface $\mathcal{V}_\infty$
given by
$y^2 - az^2 = \tilde{P}_\infty(x)$
over $k$ satisfies $\mathcal{V}_\infty(\AAA_k) \neq \emptyset$  but $\mathcal{V}_\infty(k) = \emptyset$. Such Ch\^{a}telet surfaces exist over any global field $k$ of characteristic different from $2$: see [~\cite{Poo08}, Proposition 5.1 and 11]. If $k = \QQ$ one may use the original example from ~\cite{Isk71} with $a = -1$ and $\tilde{P}_\infty(x) := (x^2 - 2)(3 - x^2)$.

Now Let  $P_\infty(w, x)$ and $P_0(w, x)$  be the homogenizations of $\tilde{P}_\infty$ and $\tilde{P}_0$. Let $\mathcal{L} = O(1, 2)$ on $\PP^1 \times \PP^1$ and define
$$s_1 := u^2P_\infty(w, x) + v^2P_0(w, x) \in \Gamma(\PP^1 \times \PP^1,\mathcal{L}^{\otimes2}) $$
where the two copies of $\PP^1$ have homogeneous coordinates $(u: v)$ and $(w: x)$ respectively. Let
$Z_1 \subset \PP^1 \times \PP^1$ be the zero locus of $s_1$. Let $F \subset \PP^1$ be the (finite) branch locus of the first projection $Z_1 \to \PP^1$. i.e.
$$ F := \left\{(u:v) \in \PP^1 | u^2P_\infty(w, x) + v^2P_0(w, x) \text{  has a multiple root} \right\}.$$
Let $\alpha_1 : \mathcal{V} \to \PP^1 \times \PP^1$ be the conic bundle given by $y^2 - az^2 = s_1$, i.e.  the conic bundle on $\PP^1 \times \PP^1$ defined by the datum $(O(1, 2),a,s_1)$.

Composing $\alpha_1$ with the first projection $\PP^1 \times \PP^1 \to \PP^1$ yields a morphism
$\beta_1 : \mathcal{V} \to \PP^1$ whose fiber above $\infty := (1 : 0)$ is the Ch\^{a}telet surface $\mathcal{V}_\infty$ defined earlier.

Now Let $C$ be a nice curve over $k$ such that $C(k)$ is finite and nonempty. Choose a dominant
morphism
$ \gamma : C \to \PP^1$, \'{e}tale above $F$, such that
$\gamma(C(k)) = \{\infty\}$.
Define  $X := \mathcal{V} \times_{\PP^1} C$ to be the fiber product with respect to the maps $\beta_1: \mathcal{V}\to \PP^1, \gamma C \to \PP^1 :$
and consider the morphisms $\alpha$ and $\beta$ as in the diagram:

$$\xymatrix{
X   \ar@/_2pc/[dd]_{\beta} \ar[d]_\alpha  \ar[r]  &   \mathcal{V} \ar[d]_{\alpha_1} \ar@/^2pc/[dd]^{\beta_1}   \\
C \times \PP^1 \ar[d]_{1^{st}} \ar[r]^{(\gamma,1)}    &   \PP^1\times \PP^1 \ar[d]_{1^{st}}   \\
C        \ar[r]^\gamma          &   \PP^1 \\
}$$

Each map labeled $1^{st}$ is the first projection.

$X$ is the variety Poonen constructed in ~\cite{Poo08},  In the same paper Poonen proves that  $X(\AAA_k)^{\acute{E}t,\Br}  \neq \emptyset$ (Theorem 8.2 in ~\cite{Poo08}) and  $X(k) = \emptyset$ (Theorem 7.2 in ~\cite{Poo08}).
We present here the proof that $X(k) = \emptyset$ since it is short and simple.
\begin{proof}
Assume $x_0\in X(k)$, we have $c_0 := \beta(x_0) \in C(k)$ but then $x \in \beta^{-1}(c_0)$. By the construction of $X$.
$\beta^{-1}(c_0)$ is isomorphic to $\beta_1^{-1}(\gamma(c_0)) = \beta_1^{-1}(\infty) \cong \mathcal{V}_\infty$
but $\mathcal{V}_\infty(k)=\emptyset$ by construction.
\end{proof}

Note that $X$ can also be considered as the variety corresponding to the datum $(O(1, 2),a,s_1)$ pulled back via $(\gamma,1)$ to $C\times \PP^1$.

\section{The Construction}

In this section we present the construction we use to explain the absence of rational points on $X$ by applying the variant of the \'{e}tale-Brauer-Manin obstruction defined in \S 1. All the notations will agree with those of the previous section.

First we shall show that almost Galois coverings behave well under pull-backs, namely:

\begin{lem}\label{l:Pull_D}
Let $X$ be a projective variety, $D \subset X$ a divisor and  $\pi:Y\to X$ a quasi-torsor under some finite $k$-group $G$ unramified outside $D$. Further assume that $D(k)=\emptyset$ and $\rho:Z \to X$ is any map. Then $\pi': Y\times_X Z \to Z$ is a $G$-quasi-torsor unramified outside $\rho^{-1}(D)$ and $\rho^{-1}(D)(k)=\emptyset$.
\end{lem}

\begin{proof}
Clear.
\end{proof}

Now let $F' := \gamma^{-1}(F) \subset C$ and denote $C' := C\backslash F'$. Note that $C'$ is a non-projective curve. Now let $D := \beta^{-1}(F')$. Note that $\infty \not \in F$ so that $ C(k) \cap F' = \emptyset$. Thus $D$ has no connected components stable under $\Gamma_k$. Therefore it is clear that $D(k) = \emptyset$. We shall use the $(\acute{E}t,\Br \thicksim D)$-obstruction defined in section \S ~\ref{sec:ram} to show that $X(k) = \emptyset$.

Now $X$ is a family indexed by $C$, of conic bundles over $\PP^1$. The fibers over any point of $C(k)$ are isomorphic to the ch\^{a}telet surface $\mathcal{V}_\infty$. All the fibers over $C'$ are smooth conic bundles (all those conic bundles has exactly 4 degenerate fibers above $\PP^1$ .

Let $E' \subset (\PP^1 \backslash F) \times (\PP^1)^4 $ be the curve defined by
$$ u^2P_\infty(w_i, x_i) + v^2P_0(w_i, x_i) = 0 , 1 \leq i \leq 4 $$
$$ (w_i:x_i) \neq (w_j:x_j) , i \neq j, 1 \leq i,j \leq 4 $$
where $(u:v)$ are the projective coordinates of $\PP^1 \backslash F$ and $(w_i:x_i), 1 \leq i \leq 4 $ are the projective coordinates of the $4$ copies of $\PP^1$. Since $\tilde{P}_\infty(x)$ and $\tilde{P}_0(x)$ are separable and coprime we have that $E'$ is a smooth connected curve and that the first projection $E' \xrightarrow{1^{st}} \PP^1 \backslash F $ gives $E'$ a structure of an \'{e}tale Galois covering of $\PP^1 \backslash F$ with an automorphism group $G = S_4$ that acts on the fibres by permuting the coordinates of $$ (w_i:x_i) , 1 \leq i \leq 4.$$
Since every birationality class of curves contains a unique projective smooth member, one can construct an $S_4$-quasi-torsor over $E \to \PP^1$ unramified outside $F$ which gives $E'$ when restricted to $\PP^1\backslash F$.

Now the $k$-twists of $E \to \PP^1 $ are classified by $H^1(k,S_4)$ which (since the action of $\Gamma_k$ on $S_4$ is trivial) coincides with the set $\Hom(\Gamma_k,S_4)/\sim$ of homomorphisms up to conjugation. More concretely, for every homomorphism $\phi : \Gamma_k \to S_4$ define $E_\phi$ to be the  $k$-form of $E$ with the Galois action that restricts to the action
$$ \sig : ((u:v),((w_1:x_1),(w_2:x_2),(w_3:x_3),(w_4:x_4))) \mapsto $$ $$ ((u:v),((w_{\phi_\sigma(1)}:x_{\phi_\sigma(1)}),(w_{\phi_\sigma(2)}:x_{\phi_\sigma(2)}),
(w_{\phi_\sigma(3)}:x_{\phi_\sigma(3)}),(w_{\phi_\sigma(4)}:x_{\phi_\sigma(4)})))^{\sig} $$
on $E'$.

Now for every $\phi : \Gamma_k \to S_4$ define $C_\phi := C\times_{\PP^1} E_\phi$ relative to $\gamma: C \to \PP^1$ and the first projection $E_\phi \to  \PP^1$ and $X_\phi := X \times_{C} C_\phi$ relative to $\beta: X \to C$ and the first projection $C_\phi \to  C$.

Note that since the maps $\gamma : C \to \PP^1$ and $E \to \PP^1$ have disjoint ramification loci we have that all $C_\phi$ are geometrically integral and so are all
the $X_\phi$.

By Lemma ~\ref{l:Pull_D} $X_\phi$ is a complete family of twists of a quasi-torsor of $X$ of degree $24$ unramified outside $D$. Since $D(k)=\emptyset$, in order to explain the fact that $X(k) = \emptyset$ it is enough to show that
$$ X_\phi(\AAA)^{\Br}   = \emptyset $$
for every $\phi \in H^1(\Gamma_k,S_4)$.

Trough out the rest of the paper we shall follow Stoll's notation from ~\cite{Sto07}
and denote by  $X(\AAA)_\bullet$ ($X(\AAA)^{\Br}_\bullet$) to denote the set $X(\AAA)$ ($X(\AAA)^{\Br}$) where the space at the infinite places is replaced with it's set of connected components.

In the rest of the paper we shall prove that if $C(k) = {C(\AAA)^{\Br} }_\bullet $ then indeed for every $\phi \in H^1(\Gamma_k,S_4)$ we have $X_\phi(\AAA)^{\Br}   = \emptyset$.

Therefore \textbf{from now on we shall assume that:} $$(*)\quad C(k) = {C(\AAA)^{\Br} }_\bullet.$$ We denote the jacobian of $C$  by $J$. We have that  $(*)$ is true if $J(k),\Sha(k,J)<\infty $ by ~\cite{Sto07} Corollary 8.1.
 Since $C(k)$ is finite it might be reasonable to expect $(*)$ to always hold.

\section{Reduction to $X_{\phi_\infty}$}

\begin{lem}\label{l:curve_BM}
For every $\phi \in H^1(k,S_4)$ we have $C_\phi(k) = {C_\phi(\AAA)^{\Br} }_\bullet $.
\end{lem}
\begin{proof}
Note that we have a non-constant map $\pi_\phi :C_\phi \to C$. The proof will rely on Stoll's results in ~\cite{Sto07}.  In ~\cite{Sto07} Stoll defines for a variety $X$ the set ${X(\AAA)^{f-ab}}_\bullet$ and proves that if $X$ is a curve then
$$ {X(\AAA)^{f-ab}}_\bullet = {X(\AAA)^{\Br} }_\bullet $$
(Corollary  7.3 ~\cite{Sto07}).

Now by Proposition 8.5 ~\cite{Sto07} and the existence of the map $\pi_\phi :C_\phi \to C$ we have that ${C(\AAA)^{f-ab}}_\bullet = {C(\AAA)^{\Br} }_\bullet = C(k)$ implies ${C_\phi(\AAA)^{\Br} }_\bullet = {C_\phi(\AAA)^{f-ab}}_\bullet = C_\phi(k)$.

\end{proof}

Denote now by $\phi_\infty \in H^1(k,S_4)$ the map $\Gamma_k \to S_4$ defined by the Galois action on the $4$ roots of $P_\infty$.

\begin{lem}
Let $\phi \in H^1(\Gamma_k, S_4)$ be such that $\phi \neq \phi_\infty$ then $C_\phi(k) = \emptyset$.
\end{lem}
\begin{proof}
Recall that $C_\phi := C\times_{\PP^1} E_\phi$. Denote $\pi_\phi : E_\phi \to \PP^1$.
Since $\phi \neq \phi_\infty$ we get that $E_\phi(k) \cap  \pi^{-1}_\phi(\infty) = \emptyset$. Now Since $\gamma(C(k)) = {\infty}$ we get that $C_\phi(k) = \emptyset$.
\end{proof}

Now denote by $\rho_\phi : X_\phi \to C_\phi$ the map defined earlier. For every $\phi \in H^1(k,S_4)$ we have
$$\rho_\phi({X_\phi(\AAA)^{\Br} }_\bullet) \subset {C_\phi(\AAA)^{\Br} }_\bullet = C_\phi(k) .$$

so we get that for $\phi \neq \phi_\infty$, $X_\phi(\AAA)^{\Br}   =\emptyset$.

\section{The proof that $X_{\phi_\infty}(\AAA)^{\Br}   =\emptyset$.}

In this section we shall prove that if $k$ does not have real places (i.e. $k$ is a function field or a totaly imaginary number field) then $X_{\phi_\infty}(\AAA)^{\Br}  =X_{\phi_\infty}(\AAA)_\bullet^{\Br}  = \emptyset$.

Let $p \in C_{\phi_\infty}(k)$. The fiber $\rho^{-1}_{\phi_\infty}(p)$ is isomorphic to the  Ch\^{a}telet surface $\mathcal{V}_\infty$. We shall denote by $\rho_p : \mathcal{V}_\infty \to X_{\phi_\infty}$ the corresponding natural isomorphism onto the fiber $\rho^{-1}_{\phi_\infty}(p)$. Recall that $\mathcal{V}_\infty$ satisfies $\mathcal{V}_\infty(\AAA)^{\Br}  = \emptyset$.

\begin{lem}\label{l:on_fibre}
Let $k$ be global field with no real embeddings. Let $x \in X_{\phi_\infty}(\AAA)_\bullet^{\Br} $. Then there exists a $p \in C_{\phi_\infty}(k)$ such that $x \in \rho_p(\mathcal{V}_\infty(\AAA)_\bullet )$.
\end{lem}

\begin{proof}
From functoriality  and Lemma ~\ref{l:curve_BM} we get
$$\rho_{\phi_\infty}(x) \in \rho_{\phi_\infty} (X_{\phi_\infty}(\AAA)_\bullet^{\Br} ) \subset  {C_{\phi_\infty}(\AAA)^{\Br} }_\bullet = C_{\phi_\infty}(k) $$
We denote $p = \rho_{\phi_\infty}(x) \in  C'_{\phi_\infty}(k)$. Now it is clear that in all but maybe the infinite places   $x \in \rho_p(\mathcal{V}_\infty(\AAA)  )$. Hence it remains to deal with the infinite places which by assumption are all complex. But since both $X_{\phi_\infty}$ and $\mathcal{V}_\infty$ are geometrically integral, taking connected components reduces $X(\mathbb{C})$ and $\mathcal{V}_\infty(\mathbb{C})$ to a single point.
\end{proof}

\begin{lem}\label{l:Br_surj}
Let  $p \in C_{\phi_\infty}(k)$ be a point. Then the map $$\rho_p^* : Br(X_{\phi_\infty}) \to Br(\mathcal{V}_\infty)$$
is surjective.

\end{lem}

We will prove Lemma ~\ref{l:Br_surj} in section ~\ref{sec:surj}.

\begin{lem}
Let $k$ be global field with no real embeddings. Then $X_{\phi_\infty}(\AAA)_\bullet^{\Br} = \emptyset$.
\end{lem}

\begin{proof}
Assume that $X_{\phi_\infty}(\AAA)_\bullet^{\Br}  \neq \emptyset$. Let $x\in X_{\phi_\infty}(\AAA)_\bullet^{\Br} $. By Lemma ~\ref{l:on_fibre} there
exists a $p \in C_{\phi_\infty}(k)$ such that $x \in \rho_p(\mathcal{V}_\infty(\AAA)_\bullet)$. Let $y \in \mathcal{V}_\infty(\AAA)_\bullet$ be such that $\rho_p(y)= x$. We shall show that $y \in \mathcal{V}_\infty(\AAA)_\bullet^{\Br}$.

Indeed let $b \in Br(\mathcal{V}_\infty)$. By Lemma ~\ref{l:Br_surj} there exists
a $\tilde{b} \in Br(X'_{\phi_\infty})$ such that $\rho_p^*(\tilde{b}) = b$. Now
$$ (y,b) = (y,\rho_p^*(\tilde{b})) = (\rho_p(y),\tilde{b}) = (x,\tilde{b}) = 0 $$
But by assumption $x\in X_{\phi_\infty}(\AAA)_\bullet^{\Br}$, so we have $(y,b) = (x,\tilde{b}) = 0$. Thus we have $y \in \mathcal{V}_\infty(\AAA)_\bullet^{\Br}  = \emptyset$ which is a contradiction.
\end{proof}

\section{The surjectivity of $\rho_p^*$}\label{sec:surj}
In this section we shall prove the statement of Lemma ~\ref{l:Br_surj}.

\begin{lem}
Let $p \in C_{\phi_\infty}(k)$ and $\rho_p : \mathcal{V}_\infty \to  X_{\phi_\infty}$ be the corresponding map as above. Then the map of Galois modules
$$\rho_p^* : \Pic(\overline{X_{\phi_\infty}}) \to \Pic(\overline{\mathcal{V}_\infty})$$

has a section.
\end{lem}

\begin{proof}
Consider the map $\phi_p: \PP^1 \to \PP^1\times C_{\phi_\infty}$ defined by $x \mapsto (x,p)$.
It is clear that the map $\rho_p:\mathcal{V}_\infty \to X_{\phi_\infty}$ comes from pulling back the conic bundle datum defining $X_{\phi_\infty}$ over $ \PP^1\times C_{\phi_\infty} $ by this map. Let $B =\PP^1\times C_{\phi_\infty}$ and consider the following commutative diagram with exact rows

$$\xymatrix{
0 \ar[r] & \bigoplus\ZZ \overline{Z}_i  \ar[r] \ar[d] & \Pic \overline{B} \oplus \bigoplus \ZZ \overline{Z}^{+}_i \oplus \bigoplus \ZZ \overline{Z}^{-}_i \ar[r] \ar[d] & \Pic \overline{X_{\phi_\infty}}
 \ar[r]_-{deg} \ar[d]^{\rho_p^*}  & \mathbb{Z}  \ar[r] \ar@{=}[d]  \ar@/_/[l] & 0
\\
0 \ar[r] & \bigoplus\ZZ \overline{W}_i \ar@/^/[u]^{s_1} \ar[r] & \Pic \overline{\PP^1}  \oplus \bigoplus \ZZ \overline{W}^{+}_i \oplus \bigoplus \ZZ \overline{W}^{-}_i \ar@/^/[u]^{s_2} \ar[r] & \Pic \overline{\mathcal{V}_\infty} \ar[r]_-{deg} &  \mathbb{Z} \ar@/_/[l]  \ar[r]& 0 \\
}$$
where  $Z$ is the degeneracy locus of $X_{\phi_\infty}$ over $B$ and $W$ is the degeneracy locus of $\mathcal{V}_\infty$ over $\PP^1$. The existence of a section for $\rho_p^*$ follows by diagram chasing and the existence of the compatible sections $s_1$ and $s_2$.

Every $W_i$ ($1 \leq i \leq 4$ ) is  a point that corresponds to a different root  $(w_i:x_i)$ of the polynomial $P_\infty(x,w)$. We can choose  $\overline{Z}_i \subset \overline{B}$ to be  Zariski closure of the zero set of  $w_ix-x_iw$,  and similarly $\overline{Z}^{\pm}_i \subset \overline{X_{\phi_\infty}}$ to be Zariski closure of the zero set of  $y \pm\sqrt{a} z, w_ix-x_iw$.

Now we define: $\overline{Z}_i = s_1(\overline{W}_i)$ and $\overline{Z}^{\pm}_i = s_2(\overline{W}^{\pm}_i)$ and the map
$s_2:\Pic \overline{\PP^1} \to  \Pic \overline{B}$ is define by the unique section of the map $\phi_p: \PP^1 \to \PP^1\times C_{\phi_\infty}$.

It is clear that $s_1$ and $s_2$ are indeed "group-theoretic" sections. To prove that $s_1$ and $s_2$ also respect the Galois action note that we can write
$$ p = (c,((x^0_1:w^0_1),(x^0_2:w^0_3),(x^0_2:w^0_3),(x^0_2:w^0_3))) \in C(k)\times_{\mathbb{P}^1(k)}E_{\phi_\infty}(k) $$
and since $\gamma(C(k))=\{\infty\}$, the four points $\{(x^0_1:w^0_1),(x^0_2:w^0_3),(x^0_2:w^0_3),(x^0_2:w^0_3)\}$ are exactly the four different roots of $P_\infty(x,w)$ .
\end{proof}

\begin{lem}[Lemma ~\ref{l:Br_surj}]
Let  $p \in C_{\phi_\infty}(k)$. Then the map $$\rho_p^* : Br(X_{\phi_\infty}) \to Br(\mathcal{V}_\infty)$$
is surjective.s
\end{lem}

\begin{proof}
Denote by $s_p: \Pic\overline{(\mathcal{V}_\infty)} \to \Pic(\overline{X_{\phi_\infty}})$ the section of
$$ \rho_p^* :\Pic(\overline{X_{\phi_\infty}}) \to \Pic(\overline{\mathcal{V}_\infty}) $$
It is clear that $s_p$ induces a section of the map
$$ \rho_p^{**} :H^1(k,\Pic(\overline{X_{\phi_\infty}})) \to H^1(k,\Pic(\overline{\mathcal{V}_\infty})) $$
Now by the Hochschild serre spectral sequence for every projective variety $X$ we have.
$$H^1(k,\Pic(\overline{X})) = \Ker[\Br X \to \Br \overline{X}] / Im[\Br k \to \Br X]$$
So if one denotes $$\Br_1(X):= \Ker[\Br X \to \Br \overline{X}]$$

We get that the map $\rho_p^*:Br_1(X_{\phi_\infty}) \to Br_1(\mathcal{V}_\infty)$ is surjective. But since $\overline{\mathcal{V}_\infty}$ is a rational surface (it is a ch\^{a}telet surface) we have $Br\overline{\mathcal{V}_\infty} = 0$, and thus  $Br_1(\mathcal{V}_\infty)=Br(\mathcal{V}_\infty)$. So we get that $\rho_p^* : Br(X_{\phi_\infty}) \to Br(\mathcal{V}_\infty)$ is surjective.
\end{proof}

\section{Obstructions applied to an open subvariety}
In this section we show that one can consider the computation done in this paper as computing the Brauer-Manin set for a non-projective variety namely the variety $X':= X\backslash D$.
Now for $\phi \in H^1(K,S_4)$ consider the map $f_\phi:X_\phi  \to X$. We shall denote
$X_\phi' =: X_\phi\backslash f_\phi^{-1}(D)$.
Note that the set $$\{f_\phi :X'_\phi\to X' | H^1(K,S_4)\}$$
is a complete set of twists of a $S_4$-torsor over $X'$.
Now we have for every  $\phi \in H^1(K,S_4)$
$$X_\phi '(\AAA)^{Br}\subset X_\phi(\AAA)^{Br}=\emptyset $$
Thus we get that
$$X'(\AAA)^{\acute{E}t,Br} = \emptyset.$$
Now we know that $D$ has no geometric connected component fixed by the Galois action
and thus by ~\cite{Sto07} Proposition 5.17. we have $D(\QQ) = D(\AAA)^{\acute{E}t,Br} = \emptyset$
To conclude we have
$$X(\QQ) = X'(\QQ) \coprod D(\QQ) \subset X'(\AAA)^{\acute{E}t,Br} \coprod D(\AAA)^{\acute{E}t,Br} = \emptyset $$

These alternative description suggests that one can study rational points on algebraic varieties by decomposing them
to a disjoint union of locally closed subvarieties.

\end{document}